\documentclass[11pt,oneside]{article}

\usepackage{amsthm, amsmath, amssymb, amsfonts, epsfig, mathdots}
\usepackage{mathtools}
\usepackage{esint}
\usepackage{epstopdf}
\usepackage{float,graphicx}
\usepackage[font=sl,labelfont=bf]{caption}
\usepackage{subcaption}
\usepackage{enumerate}
\usepackage[colorlinks=true, pdfstartview=FitV, linkcolor=blue,
            citecolor=blue, urlcolor=blue]{hyperref}
\usepackage[usenames]{color}
\definecolor{Red}{rgb}{0.7,0,0.1}
\definecolor{Green}{rgb}{0,0.7,0}
\usepackage[numbers]{natbib}
\usepackage{url}
\makeatletter\def\url@leostyle{%
 \@ifundefined{selectfont}{\def\UrlFont{\sf}}{\def\UrlFont{\scriptsize\ttfamily}}} \makeatother
\usepackage{accents, comment}
\usepackage{bbm}
\usepackage{mathrsfs}
\newtheorem{theorem}{Theorem}
\newtheorem{proposition}[theorem]{Proposition}

\theoremstyle{definition}
\newtheorem{definition}[theorem]{Definition}

\theoremstyle{remark}
\newtheorem{remark}[theorem]{Remark}

\numberwithin{equation}{section}
\numberwithin{theorem}{section}
\title{On taming Moffatt-Kimura vortices of doom in the viscous case\footnote{To appear in a special issue of PAFA dedicated to Professor Peter Constantin 
on the occasion of his 75th birthday}}

\author{
Zoran Gruji\'c\\
\vspace{0.2in}
\small UAB  \\[-0.6ex]
}

\begin{document}
\maketitle

\begin{abstract}
In this note we propose a two-layer viscous mechanism for preventing finite time singularity formation in the Moffatt-Kimura scenario of two
counter-rotating vortex rings colliding at a nontrivial angle. In the first layer the scenario is recast within the framework of the study of
turbulent dissipation based on a suitably defined `scale of sparseness' of the regions of intense fluid activity. Here it is found that the problem
is (at worst) critical, i.e., the upper bound on the scale of sparseness of the vorticity super-level sets is comparable to the lower bound on the radius of spatial
analyticity. In the second layer, an additional more subtle mechanism is identified, potentially capable of driving the scale of sparseness into 
the dissipation range and preventing the formation 
of a singularity. The mechanism originates in certain analytic cancellation properties of the vortex-stretching term in the sense of compensated compactness
in Hardy spaces which then convert information on local mean oscillations of the vorticity direction (boundedness in certain log-composite weighted
local \emph{bmo} spaces) into log-composite faster decay of the vorticity super-level sets.

\end{abstract}

\section{Introduction}

In the study of possible singularity formation in the 3D Euler and Navier-Stokes equations the scenarios based on setting up a configuration 
of vortex structures (e.g., vortex tubes or vortex rings) colliding at a nontrivial angle and engineered with an eye on maximizing the vorticity amplification primarily via
the mechanism of vortex stretching have been of a particular interest as plausible avenues to arrive at a singularity.

\medskip

Moffatt and Kimura \cite{Moffatt2019, Moffatt2019R} proposed a scenario of two counter-rotating (opposing circulation) vortex rings colliding at a nontrivial angle for which
they derived a reduced model -- a system of ODEs relating the key physical quantities of interest: the separation distance and the curvature and the radii of the vortex cores
at `tipping points'.
The mechanism is essentially Eulerian and the effects of the viscosity are 
discussed \emph{a posteriori}. Dynamics is divided in two phases, Phase I takes place before the reconnection and Phase II during reconnection.

\medskip

\begin{figure}[h!]
  \centering
     \includegraphics[scale=0.7]{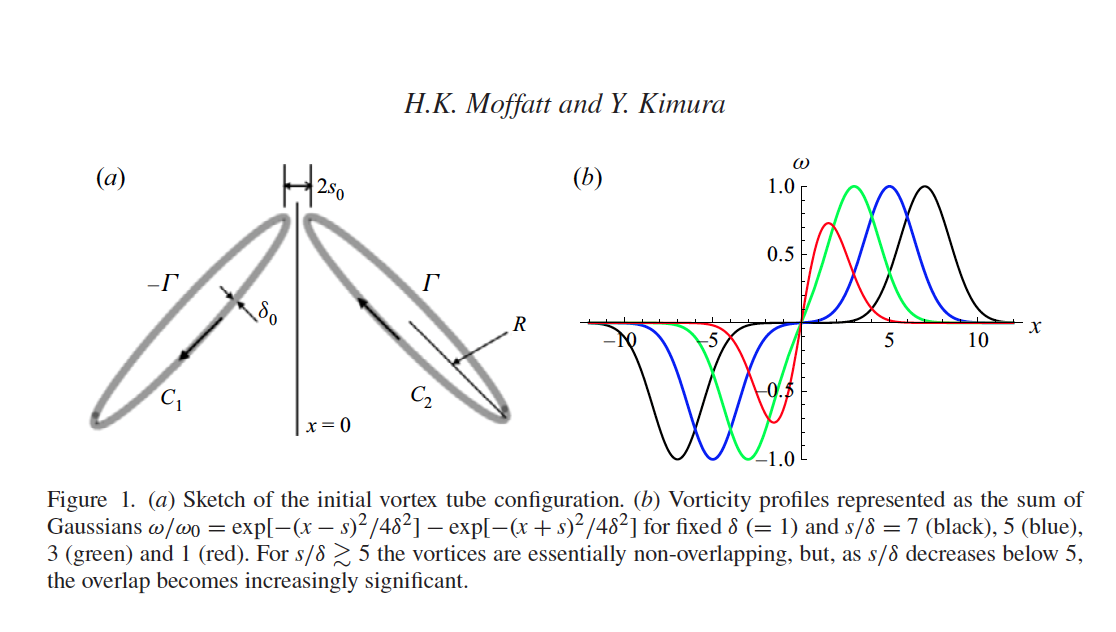}
  \caption{initial configuration \cite{Moffatt2023}}
\end{figure}

\medskip

While the scenario remains a plausible finite time blow-up scenario in the inviscid case, the situation in the viscous case, due to the process of viscous
reconnection, is less transparent. Yao and Hussain \cite{Yao2020} performed a DNS study initialized at the Moffatt-Kimura initial configuration 
and observed that -- in Phase II -- there seems to be significant flattening and stripping of the vortex cores which would be inconsistent with the model. Moreover
they argued that the formation of bridges during the viscous reconnection (Figure 2) is capable of arresting the growth of the vorticity, a mechanism that was not accounted
for in the model. Moffatt and Kimura addressed some of these issues in \cite{Moffatt2023} and -- in particular -- pointed out that according to their updated calculations
the conditions needed to attain a significant amplification of the vorticity magnitude are far beyond what can be realized either in experiments or
numerical simulations. 

\medskip

\begin{figure}[h!]
  \centering
     \includegraphics[scale=0.4]{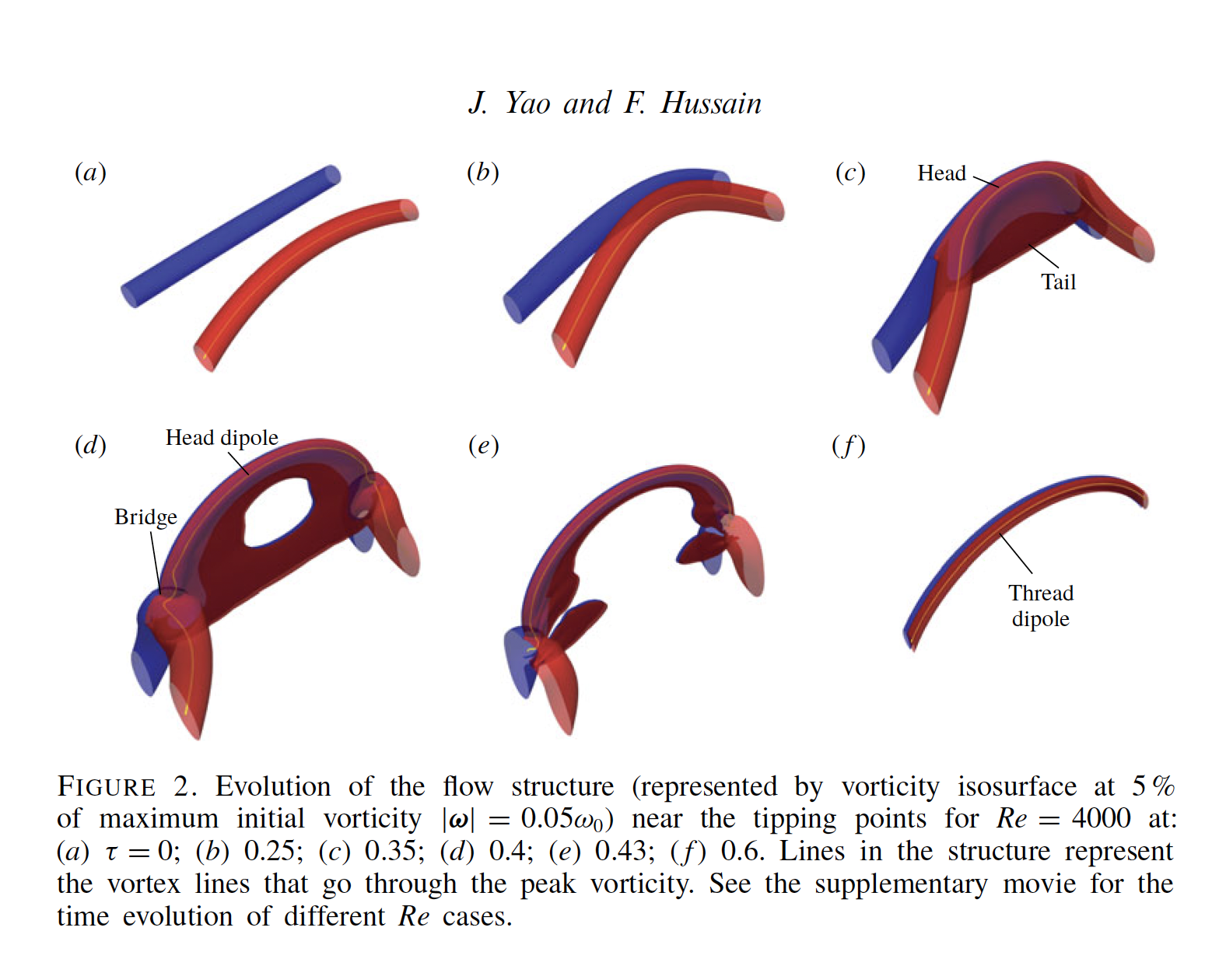}
  \caption{viscous reconnection, bridges and threads at $\emph{Re} = 4 000$ \cite{Yao2020}}
\end{figure}

\medskip

The goal of this note is to consider the problem from a slightly different angle and propose a two-layer mechanism for preventing a finite time blow-up 
in the viscous case.

\medskip

The underlying mechanism is the one of turbulent dissipation considered in the mathematical framework based on the suitably defined `scale of sparseness' of the 
vorticity super-level sets quantifying spatial intermittency. Shortly, if the scale of sparseness near a possible singular time falls below the scale of 
the radius of spatial analyticity measured in $L^\infty$ (a natural 
dissipation scale quantifying viscous smoothing), the harmonic measure maximum principle will prevent a further growth
of the vorticity magnitude contradicting a finite time
blow-up \cite{Grujic2013, Bradshaw2019, Grujic2019, Grujic2020}.

\medskip

In the first layer we observe that Moffatt-Kimura scenario is at worst critical within the aforementioned framework, i.e., the scale of sparseness encoded in
the model (diameter of the vortex cores as they rush toward a potential singularity) and the scale of
the radius of spatial analyticity (a rigorous lower bound derived from the Navier-Stokes equations) coincide. This is based on the assumption made 
in the original model where the vortex cores remain mostly compact
and nearly circular even when transitioning to Phase II. Note that within the framework, this is the geometrically worst case scenario since flattening
and stripping of the cores -- as observed in the DNS performed by Yao and Hussain -- could possibly drive the scale of sparseness into the dissipation
range. 

\medskip

Since any quantification of the possible gain of sparseness due to the formation of bridges and threads in the process of viscous reconnection is out of reach, 
we propose to explore a different route to (possibly)  braking the criticality -- the one based on the study of \emph{local mean oscillations of the vorticity direction.}

\medskip

The significance of the vorticity direction in the study of the problem of global regularity for the 3D Navier-Stokes system was first recognized
by Constantin in \cite{Constantin1994} where a singular integral representation of the stretching factor in the evolution of the vorticity magnitude was derived. 
The representation features a geometric kernel the strength of which is depleted precisely by local coherence of the vorticity direction. In a follow-up
work \cite{Constantin1993} Constantin and Fefferman showed that as long as the vorticity direction is Lipschitz in the regions of intense vorticity no finite time blow-up
can occur. Subsequently Beirao Da Veiga and Berselli reduced the Lipschitz condition to $\frac{1}{2}$-H\"older \cite{BdVBe02} which was  -- in turn -- followed up
by several works including a spatiotemporal localization of the $\frac{1}{2}$-H\"older regularity criterion \cite{Grujic2009} and a work on the role of local coherence
of the vorticity direction in the study of 3D enstrophy cascade in the NS flows \cite{Dascaliuc2013}.

\medskip

In a particular case of the critical blow-up, a significant reduction of regularity of the vorticity direction needed to prevent a finite time blow-up
is possible. More precisely Giga and Miura in \cite{Giga2011} demonstrated that in the case of Type I blow up (a blow-up rate of the $L^\infty$-norm
is bounded by the self-similar rate) it suffices that the vorticity direction is uniformly continuous in the spatial variable.

\medskip

Since the spatiotemporal scaling of Moffatt-Kimura scenario is consistent with a Type I blow-up, a potential singularity formation would then
force the regularity of the vorticity direction below that of uniform continuity. In search of a suitable functional class it is helpful to notice that
dynamics of the viscous reconnection at 
high Reynolds numbers seems to feature a creation of secondary vortices (an `avalanche') which then tend to tangle up in intricate ways
(Figure 3). Consequently -- as the flow approaches a possible singularity -- the vorticity direction is expected to exhibit 
a nontrivial oscillatory behavior and the class of local weighted spaces of functions of bounded mean oscillations seems 
a reasonable choice to encode this.

\medskip

\begin{figure}[h!]
  \centering
     \includegraphics[scale=0.4]{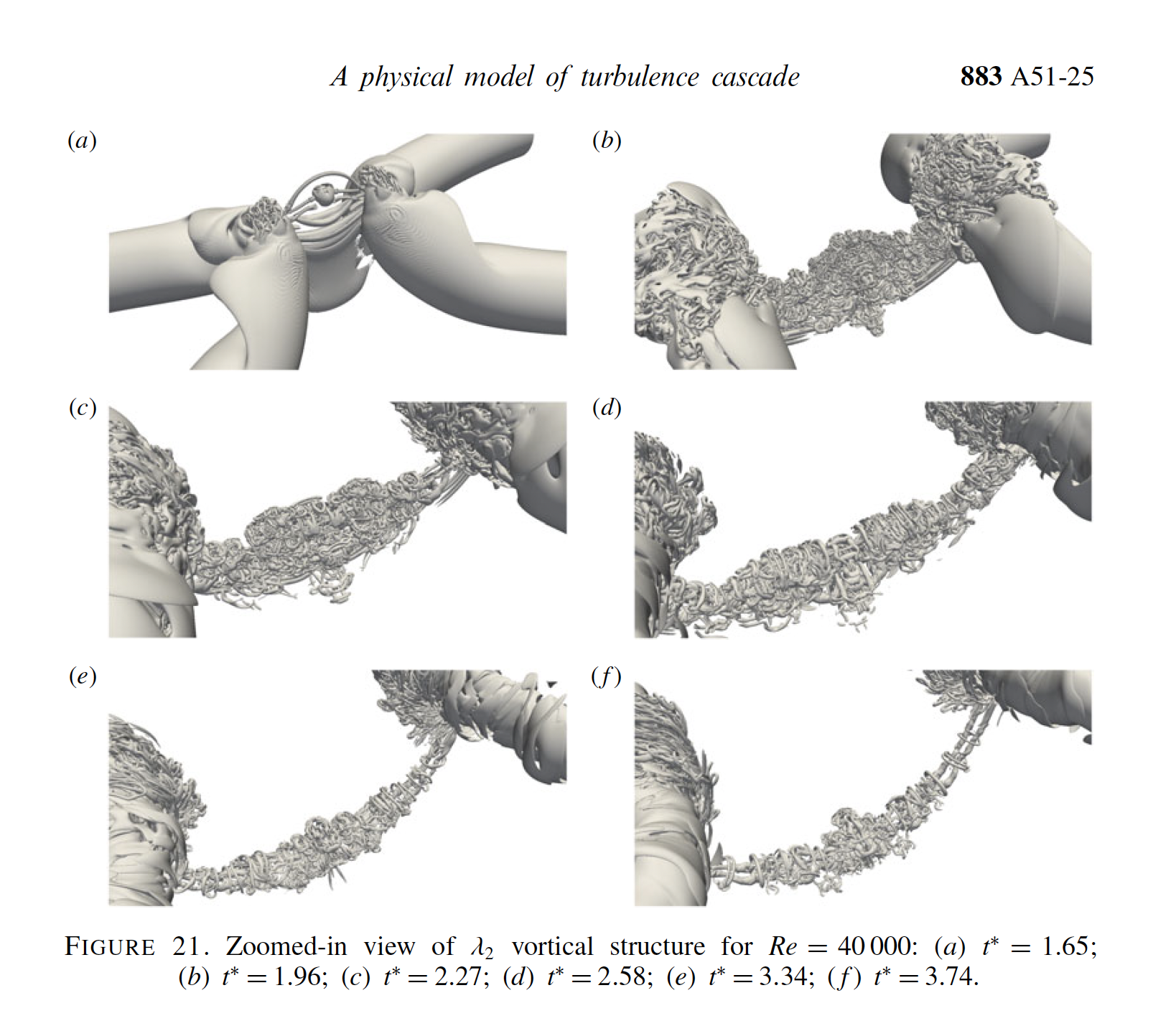}
  \caption{viscous reconnection, avalanche at  $\emph{Re} = 40 000$ \cite{Yao2020A}}
\end{figure}

\medskip

Let us denote the mean oscillation of a function $f$ over the cube $I=I(x, r)$ centered at $x$ with the side length $2r$ by
\[
 \Omega(f, I(x,r)) = \frac{1}{|I(x,r)|} \int_{I(x, r)} | f(x) - f_I| \, dx
\]
where $f_I$ is the mean value of $f$ over $I$. Assuming that $f \in L^1$ we can focus on small scales, say $0 < r < \frac{1}{2}$,
and define the local weighted space of functions of bounded mean oscillations $\widetilde{bmo}_\phi$ as follows. 
For a positive non-decreasing function $\phi$ on $(0, \frac{1}{2})$ require that the quantity
\[
 \|f\|_{\widetilde{bmo}_\phi}= \|f\|_{L^1} + \sup_{x \in \mathbb{R}^3, 0 < r , \frac{1}{2}} \frac{\Omega(f, I(x,r)) }{\phi(r)}
\]
is finite. In the case of a power weight $\phi(r) = r^\alpha$ for $0 < \alpha \le 1$ it is known that being in $\widetilde{bmo}_\phi$ 
is equivalent to being in the local $\alpha$-H\"older/Lipschitz class. 

\medskip

Since we are already in a critical scenario, the weights of the most interest here are the ones that would allow for discontinuities
(due to Giga and Miura).
One can show that $\widetilde{bmo}_\phi$ contains discontinuous functions if and only if
\[
 \int_0^\frac{1}{2} \frac{\phi(r)}{r} \, dr = \infty
\]
and since the vorticity direction (being bounded) is trivially in  $\widetilde{bmo}_1$ the sequence of weights 
\[
 \phi_k(r) = \frac{1}{\log^k(|\log r|)}
\]
($\log^l$ denotes an l-fold composition) generates a reasonable scale of spaces to focus on (there are functions that grow 
at infinity slower than any log-composite and are for all practical purposes equivalent to a constant, e.g, inverse Ackerman function; 
for the time being it seems sensible to defocus
from this).

\medskip

\begin{figure}[h!]
  \centering
     \includegraphics[scale=0.25]{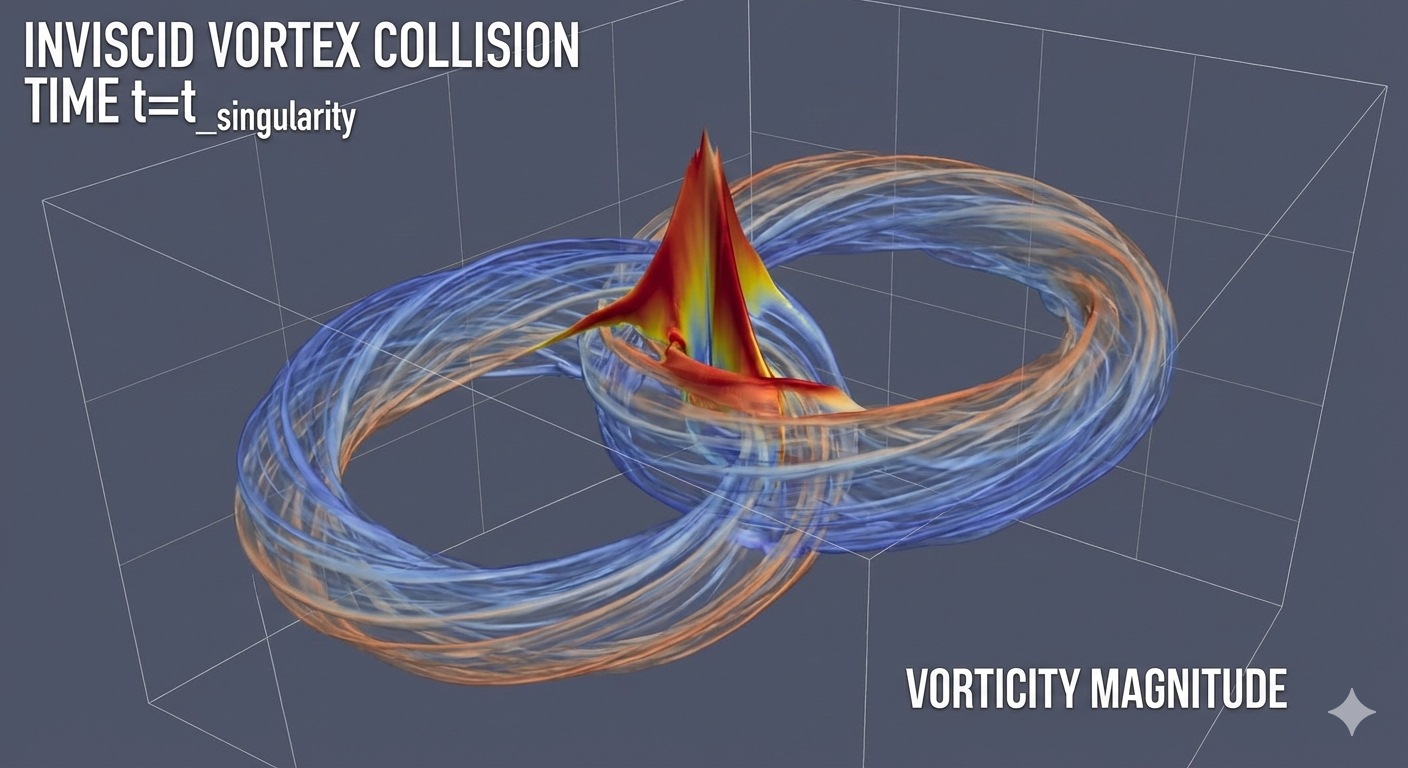}
  \caption{a sketch of the inviscid blow-up generated using Google AI tools}
\end{figure}

\medskip

This setting also points to another, more geometric/topological contrast between inviscid and viscous Moffatt-Kimura scenarios (compared to
no smoothing \emph{vs.} analytic smoothing). 
In the inviscid case -- due
to the lack of the shear stress -- the scenario of two counter-rotating vortex rings colliding at a nontrivial angle is
consistent with a formation of a `tent-like' structure (Figure 4) rising to a potential singularity which is in turn consistent with a simple jump discontinuity of the vorticity direction
field at the apex (Note that this is precisely a representative of the case where a
function would be in  $\widetilde{bmo}_1$ (in this case a trivial bound) but not in any of  $\widetilde{bmo}_\phi$ with the log-composite weights (in one
dimension, think about the signum function on  $(-r, r)$, the mean oscillation is exactly one, no dependence on $r$).
In contrast, in the viscous case, it is not unreasonable to expect that the shear stress-driven component of spatial dynamics 
near the potential singularity -- grinding of
topologically intricate local fluid layers (Figure 3) -- would be capable of taming the local mean oscillations just enough to 
nudge them over the trivial bound and into the log-composite $\widetilde{bmo}_\phi$ realm.

\medskip

In the second layer we observe that -- in the viscous Moffatt-Kimura scenario -- previous work \cite{Bradshaw2015, DFGX18} based on utilizing \emph{analytic cancellations in 
the vortex-stretching term} in the sense of compensated compactness in Hardy spaces to (logarithmically) improve upon the
\emph{a priori} bound on the $L^1$-norm of the vorticity obtained by Constantin \cite{Constantin1990} could be used to (logarithmically) improve
upon the scale of sparseness, breaking the criticality noted in the first layer. The key assumption
in \cite{DFGX18} was precisely that the vorticity direction belonged to one of the log-composite  $\widetilde{bmo}_\phi$ spaces.

\medskip

The note is organized as follows. Section 2 is devoted to the criticality scenario (layer one), Section 3 to depicting a possible road to 
logarithmic sub-criticality (layer two), and Section 4 presents a rigorous dissipation argument assuming logarithmic sub-criticality.

\section{Criticality}

On one hand, recall that in Moffatt-Kimura scenario of two counter-rotating vortex rings of the constant radii $R$ colliding at a nontrivial angle the maximal vorticity 
amplification rate at a tipping point and near the potential singular time is given (according to the reduced model) by
\[
 \mathcal{A}_\omega  = \frac{\omega_{max}(t)}{\omega_0} \approx \frac{\delta^2_0}{\delta^2(t)}
\]
where $\delta_0$ is the diameter of the vortex cores in the initial configuration, $\omega_0$ the initial axial velocity given by
$\displaystyle{\omega_0 = \frac{\Gamma}{4\pi \delta^2_0}}$ ($\Gamma$ and $-\Gamma$ are the initial circulations), $\delta(t)$ the diameter
of the vortex cores and $\omega_{max}(t)$ the vorticity maximum \cite{Moffatt2023}.
Solving for $\delta(t)$ gives 

\[
 \delta(t) \approx \frac{\Gamma^\frac{1}{2}}{\|\omega(t)\|_\infty^\frac{1}{2}}.
\]
$\delta(t)$ is a natural small (spatial) scale intrinsic to the model, the macro scale being the radii of the rings $R$. 
It is also comparable to the scale of sparseness of the vorticity super-level sets.

\medskip

 On the other hand, recall that a lower bound on the radius of spatial analyticity of the vorticity field is given by \cite{Bradshaw2019}
 
 \[ 
  \rho(t) \approx \frac{\nu^\frac{1}{2}}{\|\omega(t)\|_\infty^\frac{1}{2}}
 \]
where $\nu$ is the viscosity.

\medskip

In order to prevent the blow-up in this framework one needs $\displaystyle{\rho(t) \ge \delta(t)}$ as $t$ approaches the potential singular time and since
the dynamic quantities coincide (we are in a critical scenario) the scale of sparseness will fall into the dissipation range if

\[
 \frac{\Gamma}{\nu} \approx 1,
\]
i.e., only when the Reynolds number of the initial configuration is of order 1 (the battle of constants is not going all that great). 
However, since the scale of sparseness noted here is based on a stipulation
that the cores remain mostly compact and nearly circular, an assumption made in \cite{Moffatt2019R} but also challenged in \cite{Yao2020}
due to flattening and stripping of the cores during the viscous reconnection, there may be room for dynamic improvement.

\section{Beyond criticality?}

Let us start by observing that the scale of sparseness presented in the previous section (derived from the reduced model) can alternatively be obtained 
from the basic geometry of the
scenario and an \emph{a priori} bound on the decay of the volume of the vorticity super-level sets derived from the vorticity-velocity formulation
of the Navier-Stokes system.

\medskip

Recall that the vorticity analogue of Leray's \emph{a priori} $L^2$ bound on the velocity is Constantin's \emph{a priori} $L^1$ bound derived in 
\cite{Constantin1990}. 
The $L^1$ bound on the vorticity combined with a trivial bound on the distribution function,

\[
|\{x \in \mathbb{R}^3: \, |\omega(x,t)| > M\}| \le \frac{\|\omega(t)\|_1}{M}
\]
yields the following bound on the volume of the vorticity super-level sets relevant in the framework ($\lambda$ is a positive constant less than one the choice
of which depends on some other tuning parameters; a complete argument of this type will be given in Section 4)

\[
|\{x \in \mathbb{R}^3: \, |\omega(x,t)| > \lambda \|\omega(t)\|_\infty\}| \le \frac{c(\lambda, \|\omega_0\|_1, T)}{\|\omega(t)\|_\infty}
\]
for any $t$ in an interval $(0, T)$. Since the radii of the rings $R$ represent the macro scale in the model, the decay of the volume given
by the above inequality as the flow approaches a possible singularity will force the decay of the diameters of the vortex cores comparable
to $\displaystyle{\frac{1}{\|\omega(t)\|_\infty^\frac{1}{2}}}$, the same small scale that appeared in the previous section (as before
this is -- from the point of view of sparseness -- geometrically worst case scenario, the observed flattening and stripping
of the cores would only make it sparser).

\medskip

It is this alternative way of deriving an upper bound on the diameters of the vortex cores (from the full Navier-Stokes
system) that is amenable to further analysis -- more precisely,
a logarithmic improvement of the $L^1$ bound on the vorticity would imply a logarithmic improvement of the rate of decay of the volume of the
vorticity super-level sets which would -- in turn -- yield a logarithmic improvement on the upper bound on the vortex core diameters, breaking
the criticality. In other words, the problem is now reduced to obtaining a logarithmic improvement of the $L^1$ bound on the vorticity under
purely geometric conditions, a question of its independent interest. 

\medskip

The first contribution in this direction was presented in \cite{Bradshaw2015} where it was shown that as long as the vorticity direction is in 
$\displaystyle{\widetilde{bmo}_\frac{1}{|\log r|}}$ the vorticity magnitude will remain in $L Log L$. The proof was based on the duality
between local Hardy and BMO spaces, analytic cancellations in the vortex-stretching term in the sense of the `div-curl' lemma in
Hardy spaces, some results on multipliers in local BMO spaces and Coifman-Rochberg $BMO$ estimate on the logarithm of the \
maximal function.

\medskip

A natural follow up question was whether it was possible to generalize this result to the log-composite weights 
\[
 \phi_k(r) = \frac{1}{\log^k(|\log r|)},
\]
ideally for any positive integer $k$. Unfortunately there is an obstruction -- there is no analogue of Coifman-Rochberg estimate
in the weighted $BMO$ spaces (in the case of log-composite weights one can construct several types of counterexamples).
However, it is possible to derive a `dynamic version' of Coifman-Rochberg for a family of time-dependent functions bounded
above and below by algebraic rates modeling the build up of the vorticity field in the vicinity of 
a possible singularity (with 
no restrictions on the strength of the singularity/the rates) \cite{DFGX18}.  This was then used to obtain a generalization of the $L Log L$
result in this setting --  it was demonstrated that as long as the vorticity direction belongs to $\displaystyle{\widetilde{bmo}_{\phi_k(r)}}$
the vorticity magnitude will stay in
$L \frac{1}{\phi_k(L)}$. Similarly as in the $L^1$ case, this implies the following bound on the decay of the volume of the 
vorticity super-level sets of interest,

\begin{equation}\label{log}
|\{x \in \mathbb{R}^3: \, |\omega(x,t)| > \lambda \|\omega(t)\|_\infty\}| \le c(\lambda, \|\omega_0\|_{L \frac{1}{\phi_k(L)}}, T)  \, \frac{\phi_k (\|\omega(t)\|_\infty)}{\|\omega(t)\|_\infty}
\end{equation}
for any $t$ in an interval $(0, T)$.

\section{Dissipation via the harmonic measure maximum principle}

The purpose of this section is to provide a rigorous argument showing that as long as the bound (\ref{log}) holds, the $L^\infty$ norm of the vorticity
in the Moffatt-Kimura scenario will remain bounded and no singularity will form.

\medskip

Let us start with recalling definitions of local sparseness at scale suitable for the mathematical analysis of 
spatial intermittency \cite{Grujic2013}.

\begin{definition}
For a spatial point $x_0$ and $\delta\in (0,1)$, an open set $S$ is 1D $\delta$-sparse around $x_0$ at scale $r$ if there exists a unit vector $\nu$ such that
\begin{align*}
\frac{|S\cap (x_0-r\nu, x_0+r\nu|}{2r} \le \delta\ .
\end{align*}
\end{definition}

\begin{definition}
For a spatial point $x_0$ and $\delta\in (0,1)$, an open set $S$ is 3D $\delta$-sparse around $x_0$ at scale $r$ if
\begin{align*}
\frac{|S\cap B_r(x_0)|}{|B_r(x_0)|} \le \delta\ .
\end{align*}
\end{definition}

Note that local 3D $\delta$-sparseness at a scale implies 1D $(\delta)^\frac{1}{3}$-sparseness at the same scale; 
the converse is false.

\medskip

The sets of interest will be the vorticity super-level sets -- more precisely -- we will consider the following. 
For a time $t$ and a spatial point  $y \in \mathbb{R}^3$ there exists a vorticity component $\omega^{i, \pm}$ such that
$\omega^{i, \pm}(y, t) = |\omega(y, t)|$ (opting for the max norm in $\mathbb{R}^3$). The associated super-level set is
then defined by 
\begin{equation}\label{super}
 V^{i, \pm}_t(y) = \{x \in \mathbb{R}^3: \, \omega^{i, \pm}(x, t) > \lambda \|\omega(t)\|_\infty\}
\end{equation}
where $\lambda \in (0, 1)$ will be a suitably chosen parameter.

\medskip

Local sparseness at scale will be utilized via the harmonic measure maximum principle. Due to the rotational invariance of the equations
and locality of the argument one can always assume that the direction of local 1D sparseness is a coordinate direction. Hence
the setting of one complex variable suffices (in several complex variables, the harmonic measure would be replaced with the 
plurisubharmonic measure).
 
\begin{proposition}[\cite{Ransford1995}]\label{hm}
Let $\Omega$ be an open, connected set in $\mathbb{C}$ such that its boundary has nonzero Hausdorff dimension, and let $K$ be a Borel 
subset of the boundary. Suppose that $u$ is a subharmonic function on $\Omega$ satisfying
\begin{align*}
u(z) &\le M\ , \quad \textrm{for }z\in\Omega
\\
\limsup_{z\to\zeta} u(z) &\le m\ , \quad \textrm{for }\zeta\in K.
\end{align*}
Then
\begin{align*}
u(z)\le m \, h(z,\Omega,K) + M(1-h(z,\Omega,K)) \ , \quad \textrm{for }z\in\Omega.
\end{align*}
(Here, $h(z,\Omega,K)$ denotes the harmonic measure of $K$ with respect to $\Omega$, evaluated
at $z$.)
\end{proposition}

\medskip

The following extremal property of the harmonic measure in the unit disc $\mathbb{D}$ will be helpful in the calculation.

\medskip

\begin{proposition}[\cite{Solynin1997}]\label{sol}
Let $\alpha$ be in $(0, 1)$, $K$ a closed subset of $[-1,1]$ 
such that $|K| = 2\alpha$,
and suppose that the origin is in $\mathbb{D} \setminus K$. Then
\[
 h(0,\mathbb{D},K) \ge h(0,\mathbb{D}, K_\alpha) =
 \frac{2}{\pi} \arcsin \frac{1-(1-\alpha)^2}{1+(1-\alpha)^2}
\]
where $K_\alpha = [-1, -1+\alpha] \cup [1-\alpha, 1]$.
\end{proposition}

\medskip

Lastly, we need a local-in-time lower bound on the radius of spatial analyticity of the vorticity field in $L^\infty$.

\medskip

\begin{theorem}[\cite{Bradshaw2019}]\label{radius}
Let $\omega_0 \in L^\infty$ and $M \in (1, \infty)$. Then there exists a unique mild solution to the vorticity formulation
of the 3D HD NS system $\omega \in C_w([0, T], L^\infty)$ where
\[
  T \ge \frac{1}{c_1(M)} \frac{\nu}{\|\omega_0\|_\infty}
\]
and for any $t \in (0, T]$ the solution $\omega$ is the $\mathbb{R}^3$-restriction of a holomorphic function $\omega$ defined in
\[
 \Omega_t = \biggl\{ x+iy \in \mathbb{C}^3: \, |y|  <  \frac{1}{c_2(M)} t^\frac{1}{2} \biggr\}
\]
satisfying $\|\omega(t)\|_{L^\infty(\Omega_t)}\le M \|\omega_0\|_\infty$.
\end{theorem}

\medskip

\medskip

In what follows the choice of the cut-off parameter $\lambda$ and the sparseness parameter $\delta$ will have to be consistent 
with the choice to be made in the application of the harmonic measure maximum principle -- with that in mind, set 
$\delta=\frac{3}{4}$ and $\lambda=\frac{1}{2M}$ where $M$ is the solution to 
\begin{equation}\label{M}
 \frac{1}{2} h^* + (1-h^*) M = 1 \ \ \mbox{and} \ \ h^*=\frac{2}{\pi} \arcsin \frac{1-\frac{3}{4}^\frac{2}{3}}{1+\frac{3}{4}^\frac{2}{3}}
\end{equation}
and -- at the same time -- $M$ to be used in the application of Theorem \ref{radius} (it is easy to check that $M>1$).

\medskip

In addition, it will be convenient to have a designation of `escape time'. Let $\omega \in C_w\bigl( [0, T^*), L^\infty\bigr)$
where $T^*$ is the first blow-up time. A time $t$ is an \emph{escape time} provided $\|\omega(s)\|_\infty > \|\omega(t)\|_\infty$
for any $s \in (t, T^*)$ (local-in-time well-posedness in $L^\infty$ implies that for any level there exists a
unique escape time).

\medskip

All tools for the proof of our result are now collected.

\begin{theorem}\label{main}

Consider Moffatt-Kimura scenario in the viscous case and suppose that the \emph{a priori} bound (\ref{log}) holds. 
Then the vorticity magnitude remains bounded and a finite time blow-up is avoided. 

\end{theorem}

\begin{proof}

Let $M$ be as in (\ref{M}) and $t$ an escape time. Solve the 3D NS system at $t$ according
to Theorem \ref{radius} with the same choice of $M$, and let $s=t+T_t$ where $T_t$ is the maximal time of existence
guaranteed by the theorem. The the solution at $s$ is analytic with the uniform radius of analyticity of at least

\[
 \frac{1}{c_3} \frac{\nu^\frac{1}{2}}{\|\omega(t)\|_\infty^\frac{1}{2}};
\]
since $t$ is an escape time the lower bound at $s$ could be replaced with

\begin{equation}\label{rho}
\rho_s =  \frac{1}{c_3} \frac{\nu^\frac{1}{2}}{\|\omega(s)\|_\infty^\frac{1}{2}}.
\end{equation}

\medskip

Next, note that the bound (\ref{log}) (taking any $T$ greater than the potential blow-up time) and the geometry
of Moffatt-Kimura scenario imply that for any $y \in \mathbb{R}^3$ 
the scale of 3D sparseness 
of the super-level sets $V^{i, \pm}_s(y)$ defined in (\ref{super}) is given by
\begin{equation}\label{r}
  r_s =  c_4(\|\omega_0\|_{L \frac{1}{\phi_k(L)}}, R) \, \biggl(\frac{\phi_k (\|\omega(s)\|_\infty)}{\|\omega(s)\|_\infty}\biggr)^\frac{1}{2}
\end{equation}
(with the choice of the parameters made in (\ref{M}).

\medskip

At this point we make a choice of the escape time $t$ to be an escape time for which $\rho_s \ge r_s$;
then $\omega(s)$ is both analytic and 1D $\frac{3}{4}^\frac{1}{3}$-sparse at scale $r_s$.

\medskip

Let $x_0 \in \mathbb{R}^3$ be arbitrary. We aim to show $|\omega(x_0, s)| \le \|\omega(t)\|_\infty$; this would
contradict $t$ being an escape time and conclude the argument.

\medskip

Due to the translational and rotational invariance of the equations, we can assume that $x_0$ is the origin and 
the direction of local 1D sparseness is the coordinate direction $e_1$. Immerse $e_1$ in the complex plane
and consider
\[
 D_{r_s} = \{ z \in \mathbb{C}: \, |z|<r_s\}.
\]

Since each $\omega_i^\pm(\cdot, s)$ is subharmonic on $D_{r_s}$ the stage is set for an application of the 
harmonic measure maximum principle (Proposition \ref{hm}). Let $\omega_i^\pm$ be the local (at 0) maximal 
component, i.e.,  $\omega_i^\pm(0, s) = |\omega(0, s)|$, and recall that the corresponding super-level set
$V_s^{i, \pm} = V_s^{i, \pm}(0)$ is 1D $\frac{3}{4}^\frac{1}{3}$-sparse at scale $r_s$.

\medskip

Next, define a compact set $K$ to be the complement in $[-r_s, r_s]$ of the set $V_s^{i, \pm} \cap (-r_s, r_s)$,
and note that -- due to sparseness -- $|K| \ge 2r_s \bigl( 1-\frac{3}{4}^\frac{1}{3}\bigr)$. 

\medskip

In order to apply the estimate on the harmonic measure given in Proposition \ref{sol}, we need to consider 
the case $0 \in K$ separately. This is straightforward since in this case
\[
 |\omega(0, s)| = \omega_i^\pm(0, s) \le \frac{1}{2M} \|\omega(s)\|_\infty \le \frac{1}{2} \|\omega(t)\|_\infty
\]
and we obtain a contradiction with $t$ being an escape time (the last inequality follows from the bound 
on the vorticity given in Theorem \ref{radius}).

\medskip

In the case $0 \notin K$, since the harmonic measure is invariant with respect to $z \to \frac{1}{r_s} z$, 
Proposition \ref{sol} yields
\[
 h(0, D_{r_s}, K) \ge \frac{2}{\pi} \arcsin \frac{1-\frac{3}{4}^\frac{2}{3}}{1+\frac{3}{4}^\frac{2}{3}} 
\]
which is precisely $h^*$ in (\ref{M}). Hence, Proposition \ref{hm} implies the following bound 
\[
 |\omega(0, s)|=\omega_i^\pm(0, s) \le h^* \frac{1}{2}\|\omega(t)\|_\infty+(1-h^*)M\|\omega(t)\|_\infty = \|\omega(t)\|_\infty
\]
by the choice of parameters made in (\ref{M}), and we obtain a contradiction again. This completes the proof.

\end{proof}

\section{Conclusion}

The main goal of this note was to propose a viscous two-layer mechanism for avoiding a finite time singularity formation
in Moffatt-Kimura scenario of two counter-rotating vortex rings colliding at a non-trivial angle. 

\medskip

It is worth emphasizing that being in the viscous case is crucial for both layers.
The first layer is built on having a lower bound on the radius of spatial analyticity stemming from local-in-time
analytic smoothing of the Navier-Stokes system. In the Euler case there is no smoothing and the only way to generate a
local-in-time analytic solution is to start with the analytic initial data but even then the lower bound is inadequate
(decreasing instead of increasing).
The second layer hinges on a stipulation that the shear stress component of topologically intricate dynamics 
of the viscous reconnection -- especially at high Reynolds numbers (Figure 3) -- is capable of slowing down the local oscillations
of the vorticity direction just enough to tip them over from being bounded in mean (a trivial bound) to having
a log-composite decay in mean where the number of composites can be arbitrary large. 
This is in contrast with the inviscid case where -- due to the lack of the shear stress -- the vorticity direction in 
Moffatt-Kimura scenario is expected to exhibit a simple jump discontinuity at the apex (Figure 4) and no decay of the local mean
oscillations saturating the trivial bound.

\medskip

As far as rigorous justification of the mechanism goes -- mathematical analysis 
based on the vorticity-velocity formulation of the Navier-Stokes system (rather than the reduced ODE 
model) taking into account only the bare-bones geometry identifying the scenario -- what remains 
to be shown is that the viscous mechanics will indeed yield a log-composite bound on the rate of decay 
of the bounded mean oscillations of 
the vorticity direction field. 
This is challenging but worth pursuing since the mechanism could potentially be adopted to other 
physically plausible scenarios for finite time singularity formulation based on amplifying the vorticity magnitude via the process of
vortex stretching. This last piece of the puzzle is a focus of current research.

\medskip

One last comment is to note that if this mechanism was realized in a particular flow, it might be impossible 
to detect in a computational simulation. Namely, recall that the condition for the crossover into the sub-critical 
regime is that a lower bound on the radius
of spatial analyticity $\rho_s$ (\ref{rho}) dominates an upper bound on the scale of sparseness $r_s$ (\ref{r}))
near the potential singular time
which unravels to
\[
 \log^{k+1} \bigl(\|\omega(s)\|_\infty)\bigr) \approx \frac{1}{\nu^\frac{1}{2}}.
\]
Consequently the level of
the vorticity magnitude at which the crossover into the sub-critical regime took place ( $\|\omega(s)\|_\infty$ )
would be of the order of a tetration of height $k+1$ (with base $e$ and $\nu^{-\frac{1}{2}}$ on top), and the dissipation scale to which 
the flow would have to be resolved 
would essentially be of the order of its reciprocal,
\[
   \frac{\nu^\frac{1}{2}}{\|\omega(s)\|_\infty^\frac{1}{2}} \approx \sqrt{ \frac{\nu}{e^{e^{\iddots^{e^{(\nu^{-\frac{1}{2})}}}}}}},
\]
which -- for $k$ large enough -- would be outside of the range of computational feasibility.
In other words -- in this scenario -- a simulation run would likely
remain in the critical regime in which the battle of constants was already lost.

\begin{remark}
The analysis in this note so far did not take into account the blow-up rate intrinsic to the Moffatt-Kimura scenario,
primarily to make it applicable to a wider class of singularity formation scenarios based on the amplification of the vorticity magnitude
via the mechanics of vortex stretching. In the Moffatt-Kimura scenario \cite{Moffatt2019}, the blow-up rate is consistent with the 
formation of the critical singular profile,
\[
 |u(x)| \approx \frac{1}{|x|}, \ \ \ \ \  |\omega(x)| \approx \frac{1}{|x|^2}
\]
which is consistent with $\omega \in L^\infty \bigl((0, T^*), L^{3/2, \infty}\bigr)$ where $L^{3/2, \infty}$ is a Lorentz space
(weak $L^{3/2}$). Denoting $\sup_{t \in (0, T^*)} \|\omega(t)\|_{L^{3/2, \infty}}$ by $c_0$, we have
\[
|\{x \in \mathbb{R}^3: \, |\omega(x,t)| > M\}| \le \frac{c_0^{3/2}}{M^{3/2}}
\]
which -- in turn -- implies
\[
|\{x \in \mathbb{R}^3: \, |\omega(x,t)| > \lambda \|\omega(t)\|_\infty\}| \le \frac{c(\lambda, c_0)}{\|\omega(t)\|_\infty^{3/2}}.
\]
Pairing this with the basic geometry of the Moffatt-Kimura scenario yields an upper bound on the local scale of
sparseness of the order of
\[
 \frac{1}{\|\omega(t)\|_\infty^\frac{3}{4}}
\]
which falls deeply (algebraically) into the dissipation range delineated by the lower bound on the radius of spatial analyticity
of the order of
\[
 \frac{1}{\|\omega(t)\|_\infty^\frac{1}{2}}.
\]
In other words, the strong anisotropy of the Moffatt-Kimura scenario converts its `isotropic criticality' ($|\omega(x)| \approx \frac{1}{|x|^2}$)
into sub-criticality (within the scale of sparseness-analyticity radius tug-of-war framework) and \emph{the potential blow-up is averted 
by an algebraic margin}. Somewhat paradoxically, the very mechanism responsible for the amplification of the vorticity magnitude
is responsible for driving the flow into the dissipation range.
\end{remark}

\section{Acknowledgments}

The work is supported in part by the National Science Foundation grant DMS 2307657. The author would like to 
express his gratitude to Professor Peter Constantin for opening the door (quite emphatically) into the realm of 
the mathematical study of the interplay between the geometric properties of 3D incompressible flows 
and their regularity as well as for his support over the years. Thanks to the referee for their comments that 
improved clarity of the exposition.

\def\cprime{$'$}

\end{document}